\documentclass[a4paper]{article}

\usepackage{amsmath}
\usepackage{amssymb}
\usepackage{amsthm}
\usepackage{enumitem}
\usepackage{geometry}
\usepackage{tikz-cd}

\theoremstyle{definition}
\newtheorem{dfn}{Definition}[section]

\theoremstyle{plain}
\newtheorem{lem}[dfn]{Lemma}
\newtheorem{prop}[dfn]{Proposition}
\newtheorem{thm}[dfn]{Theorem}
\newtheorem{cor}[dfn]{Corollary}
\newtheorem*{thm*}{Theorem}

\theoremstyle{remark}
\newtheorem{rem}[dfn]{Remark}

\numberwithin{equation}{section}

\AtEndDocument{%
  \par
  \medskip
  \begin{tabular}{@{}l@{}}%
    Naoya Hiramae\\
    \textsc{Department of Mathematics, Kyoto University}\\
    Kitashirakawa Oiwake-cho, Sakyo-ku, 606-8502, Kyoto\\
    \textit{E-mail address}: \texttt{hiramae.naoya.58r@st.kyoto-u.ac.jp}
  \end{tabular}
}

\title{$\tau$-Tilting finiteness of group algebras over generalized symmetric groups}
\author{Naoya Hiramae}
\date{}

\begin{document}

\maketitle

\renewcommand{\thefootnote}{\fnsymbol{footnote}}
\renewcommand{\thefootnote}{\arabic{footnote}}

\begin{abstract}
    In this paper, we show that weakly symmetric $\tau$-tilting finite algebras have positive definite Cartan matrices, which implies that we can prove $\tau$-tilting infiniteness of weakly symmetric algebras by calculating their Cartan matrices. Similarly, we obtain the condition on Cartan matrices that selfinjective algebras are $\tau$-tilting infinite. By applying this result, we show that a group algebra of $(\mathbb{Z}/p^l\mathbb{Z})^n\rtimes H$ is $\tau$-tilting infinite when $p^l\geq n$ and $\#\mathrm{IBr}\,H\geq\min\{p,3\}$, where $p>0$ is the characteristic of the ground field, $H$ is a subgroup of the symmetric group $\mathfrak{S}_n$ of degree $n$, the action of $H$ permutes the entries of $(\mathbb{Z}/p^l\mathbb{Z})^n$, and $\mathrm{IBr}\,H$ denotes the set of irreducible $p$-Brauer characters of $H$. Moreover, we show that under the assumption that $p^l\geq n$ and $H$ is a $p'$-subgroup of $\mathfrak{S}_n$, $\tau$-tilting finiteness of a group algebra of a group $(\mathbb{Z}/p^l\mathbb{Z})^n\rtimes H$ is determined by its $p$-hyperfocal subgroup.
\end{abstract}

\section{Introduction}

Throughout this paper, $k$ always denotes an algebraically closed field with prime characteristic $p$ and algebras means finite dimensional $k$-algebras. Modules are always left and finitely generated.\\
\indent Demonet, Iyama and Jasso \cite{DIJ} introduced a new class of algebras, \textit{$\tau$-tilting finite algebras}, and showed that $\tau$-tilting finiteness of algebras relates to brick finiteness and functorially finiteness of all the torsion classes. Miyamoto and Wang \cite{MW} posed the question whether derived equivalences preserve $\tau$-tilting finiteness over symmetric algebras, and it was shown to be true over symmetric algebras of polynomial growth \cite{MW} and over Brauer graph algebras \cite{AAC}. This question is important because $\tau$-tilting finiteness of symmetric algebras implies tilting-discreteness if $\tau$-tilting finiteness is invariant under derived equivalences over symmetric algebras (see \cite[Corollary 2.11]{AM}). Therefore, it is meaningful to consider $\tau$-tilting finiteness of algebras and many researchers have studied the $\tau$-tilting finite algebras (for example, \cite{Ad, AAC, AH, AHMW, AW, ALS, KK1, KK2, K, MS, MW, Mi, Mo, P, STV, W1, W2, W3, W4, Z}).\\
\indent Our ultimate goal is to clarify what kinds of subgroups control $\tau$-tilting finiteness of blocks of group algebras of finite groups, just as representation types of blocks could be classified in terms of defect groups (see Theorem \ref{repblock}). In \cite{HK}, we conjectured that $\tau$-tilting finiteness of a group algebra of a finite group $G$ is determined by a so-called \textit{$p$-hyperfocal subgroup} of $G$, and showed that this conjecture holds in the case $G=P\rtimes H$, where $P$ is an abelian $p$-group and $H$ is an abelian $p'$-group acting on $P$. In this paper, to provide another example in which the conjecture holds, we investigate $\tau$-tilting finiteness of a group algebra $k[(\mathbb{Z}/m\mathbb{Z})^n\rtimes H]$, where $H$ is a subgroup of the symmetric group $\mathfrak{S}_n$ of degree $n$ and the action of $H$ on $(\mathbb{Z}/m\mathbb{Z})^n$ is the permutation of entries. Let $m:=m'p^l$ with a positive integer $m'$ coprime to $p$ and a nonnegative integer $l$.\\
\indent We first discuss the relationship between Cartan matrices and $\tau$-tilting finiteness of selfinjective algebras, and show the following propositions:\\

\begin{prop}[See Proposition \ref{tinf}]\label{itinf}
    Assume that $\Lambda$ is a weakly symmetric algebra.
    \begin{enumerate}
        \setlength{\parskip}{0cm}
        \setlength{\itemsep}{0cm}
        \item If $\Lambda$ is $\tau$-tilting finite, then the Cartan matrix of $\Lambda$ is positive definite.
        \item If $\Lambda$ is $g$-tame, then the Cartan matrix of $\Lambda$ is positive semidefinite.\\
    \end{enumerate}
\end{prop}

\begin{prop}[See Proposition \ref{tinfsi}]\label{itinfsi}
    Assume that $\Lambda$ is a selfinjective algebra with $t$ simple modules. Let $C_\Lambda$ be the Cartan matrix of $\Lambda$ and $\nu\in\mathfrak{S}_t$ the Nakayama permutation of $\Lambda$. If there exists a nonzero vector $v\in\mathbb{Z}^t$ such that $v^\top C_\Lambda v\leq 0$ and $v$ is invariant under the action of $\nu$, the permutation of entries, then $\Lambda$ is $\tau$-tilting infinite.\\
\end{prop}

\noindent Thanks to the above propositions, we can show $\tau$-tilting infiniteness of algebras by calculating Cartan matrices in some cases. Next, we construct the selfinjective quotient $\mathcal{C}\rtimes H$ of $k[(\mathbb{Z}/m\mathbb{Z})^n\rtimes H]$ and calculate the Cartan matrix and the Nakayama permutation of $\mathcal{C}\rtimes H$. Moreover, we show that $\mathcal{C}\rtimes H$ satisfies the assumption of Proposition \ref{itinfsi} when $p^l\geq n$ and $\#\textrm{IBr}\,H\geq\min\{p,3\}$, where $\textrm{IBr}\,H$ is the set of irreducible $p$-Brauer characters of $H$ (note that there exists a bijection between $\mathrm{IBr}\,H$ and the set of isoclasses of simple $kH$-modules). Hence, we obtain the following theorem:\\

\begin{thm}[See Corollary \ref{main1}]\label{intro}
    If $p^l\geq n$ and $\#\mathrm{IBr}\,H\geq\min\{p,3\}$, then $k[(\mathbb{Z}/m\mathbb{Z})^n\rtimes H]$ is $\tau$-tilting infinite.\\
\end{thm}

\indent By computing the $p$-hyperfocal subgroup of $(\mathbb{Z}/m\mathbb{Z})^n\rtimes H$ in the case $H$ is a $p'$-group, we get the main result, which provides a new example verifying the conjecture posed in \cite{HK}.\\

\begin{thm}[See Theorem \ref{main2}]
    Assume that $p^l\geq n$ and $H$ is a $p'$-subgroup of $\mathfrak{S}_n$. Denote by $R$ the $p$-hyperfocal subgroup of $(\mathbb{Z}/m\mathbb{Z})^n\rtimes H$. Then $k[(\mathbb{Z}/m\mathbb{Z})^n\rtimes H]$ is $\tau$-tilting finite if and only if $R$ has rank $\leq1$.\\
\end{thm}

\noindent\textbf{Notation.} Unless otherwise specified, a symbol $\otimes$ means a tensor product over $k$. For an algebra $\Lambda$, we denote the opposite algebra of $\Lambda$ by $\Lambda^{\mathrm{op}}$, the category of $\Lambda$-modules by $\Lambda\textrm{-mod}$, the full subcategory of $\Lambda\textrm{-mod}$ consisting of all projective $\Lambda$-modules by $\Lambda\textrm{-proj}$, the homotopy category of bounded complexes of projective $\Lambda$-modules by $K^b(\Lambda\textrm{-}\mathrm{proj})$, and the derived category of bounded complexes of $\Lambda$-modules by $D^b(\Lambda\textrm{-mod})$. For $M\in\Lambda\textrm{-mod}$, we denote the number of nonisomorphic indecomposable direct summands of $M$ by $|M|$, the $k$-dual of $M$ by $DM$ and the Auslander-Reiten translate of $M$ by $\tau M$ (see \cite{ARS} for the definition and more details). For a complex $T$, we denote a shifted complex by $i$ degrees of $T$ by $T[i]$.\\

\section{Preliminaries}

In this section, let $\Lambda$ be an algebra and we recall the basic materials of $\tau$-tilting theory and $\tau$-tilting finite algebras.\\

\subsection{$\tau$-Tilting theory}

Adachi, Iyama and Reiten \cite{AIR} introduced $\tau$-tilting theory and gave a correspondence among support $\tau$-tilting modules, two-term silting complexes and functorially finite torsion classes. In addition, it has been found that support $\tau$-tilting modules corresponds bijectively to other many objects such as left finite semibricks \cite{A1}, two-term simple-minded collections \cite{KY}, intermediate $t$-structures of length heart \cite{BY}, and more. In this subsection, we only explain the main results in \cite{AIR}.\\

\begin{dfn}
    A module $M\in\Lambda\textrm{-}\mathrm{mod}$ is a \textit{support $\tau$-tilting module} if $M$ satisfies the following conditions:
    \begin{enumerate}
        \setlength{\parskip}{0cm}
        \setlength{\itemsep}{0cm}
        \item $M$ is $\tau$-rigid, that is, $\mathrm{Hom}_\Lambda(M,\tau M)=0$.
        \item There exists $P\in\Lambda\textrm{-}\mathrm{proj}$ such that $\mathrm{Hom}_{\Lambda}(P,M)=0$ and $|P|+|M|=|\Lambda|$.
    \end{enumerate}
    When we specify the projective $\Lambda$-module $P$ in (b), we write a support $\tau$-tilting module $M$ as a pair $(M,P)$.\\
\end{dfn}

\begin{prop}[{\cite[Theorem 2.7]{AIR}}]
    The set of isoclasses of basic support $\tau$-tilting $\Lambda$-modules is a partially ordered set with respect to the following relation:
    $$M\geq M'\Leftrightarrow there\;exists\;a\;surjective\;homomorphism\;from\;a\;direct\;sum\;of\;copies\;of\;M\;to\;M'.$$
\end{prop}

\vspace{3mm}

\begin{dfn}
    A complex $T\in K^b(\Lambda\textrm{-}\mathrm{proj})$ is \textit{tilting} (resp. \textit{silting}) if $T$ satisfies the following conditions:
    \begin{enumerate}
        \setlength{\parskip}{0cm}
        \setlength{\itemsep}{0cm}
        \item $T$ is \textit{pretilting} (resp. \textit{presilting}), that is, $\mathrm{Hom}_{K^b(\Lambda\textrm{-}\mathrm{proj})}(T,T[i])=0$ for all integers $i\neq 0$ (resp. $i>0$).
        \item The full subcategory $\mathrm{add}\,T$ of $K^b(\Lambda\textrm{-}\mathrm{proj})$ consisting of all complexes isomorphic to direct sums of direct summands of $T$ generates $K^b(\Lambda\textrm{-}\mathrm{proj})$ as a triangulated category.\\
    \end{enumerate}
\end{dfn}

\begin{prop}[{\cite[Theorem 2.11]{AI}}]
    The set of isoclasses of basic silting complexes in $K^b(\Lambda\textrm{-}\mathrm{proj})$ is a partially ordered set with respect to the following relation:
    $$T\geq T'\Leftrightarrow\mathrm{Hom}_{K^b(\Lambda\textrm{-}\mathrm{proj})}(T,T'[i])=0\;for\;all\;integers\;i>0.$$
\end{prop}

\vspace{3mm}

\indent We say that a complex $T\in K^b(\Lambda\textrm{-}\mathrm{proj})$ is \textit{two-term} if its $i$-th term $T^i$ vanishes for all $i\neq-1,0$. We denote by $\mathrm{s}\tau\textrm{-}\mathrm{tilt}\,\Lambda$ the set of isoclasses of basic support $\tau$-tilting $\Lambda$-modules and by $2\textrm{-}\mathrm{silt}\,\Lambda$ the set of isoclasses of basic two-term silting complexes in $K^b(\Lambda\textrm{-}\mathrm{proj})$.\\

\begin{thm}[{\cite[Theorem 3.2]{AIR}}]
    There exists an isomorphism as partially ordered sets
    $$
    \begin{tikzcd}[row sep = 1mm]
        \mathrm{s}\tau\textrm{-}\mathrm{tilt}\,\Lambda \rar[leftrightarrow, "\sim"] \dar[phantom, "\rotatebox{90}{$\in$}"] & 2\textrm{-}\mathrm{silt}\,\Lambda, \dar[phantom, "\rotatebox{90}{$\in$}"]\\
        (M,P) \rar[mapsto] & (P^M_1\oplus P\xrightarrow{(f\;0)}P^M_0), \\
        \mathrm{Cok}\,g \rar[mapsfrom] & (P^{-1}\xrightarrow{g}P^0), 
    \end{tikzcd}
    $$
    where $P^M_1\xrightarrow{f}P^M_0\rightarrow M\rightarrow0$ is the minimal projective presentation of $M$.\\
\end{thm}

\begin{dfn}
    A full subcategory $\mathcal{T}$ of $\Lambda\textrm{-mod}$ is a \textit{torsion class} if it is closed under taking factors and extensions, that is, for any short exact sequence $0\rightarrow L\rightarrow M\rightarrow N\rightarrow 0$ in $\Lambda\textrm{-mod}$, the following hold:
    \begin{enumerate}
        \setlength{\parskip}{0cm}
        \setlength{\itemsep}{0cm}
            \item $M\in\mathcal{T}$ implies $N\in\mathcal{T}$.
            \item $L,N\in\mathcal{T}$ implies $M\in\mathcal{T}$.\\
    \end{enumerate}
\end{dfn}

\begin{dfn}
    A full subcategory $\mathcal{A}$ of $\Lambda\textrm{-mod}$ is \textit{functorially finite} if for any $M\in\Lambda\textrm{-}\mathrm{mod}$, there exist $X,\,Y\in\mathcal{A}$ and $f:M\rightarrow X,\,g:Y\rightarrow M\in\Lambda\textrm{-}\mathrm{mod}$ such that $-\circ f:\mathrm{Hom}_\Lambda(X,\mathcal{A})\rightarrow\mathrm{Hom}_\Lambda(M,\mathcal{A})$ and $g\circ -:\mathrm{Hom}_\Lambda(\mathcal{A},Y)\rightarrow\mathrm{Hom}_\Lambda(\mathcal{A},M)$ are both surjective.\\
\end{dfn}

We denote by $\textrm{f-tors}\,\Lambda$ the set of functorially finite torsion classes in $\Lambda\textrm{-mod}$.\\

\begin{thm}[{\cite[Theorem 2.7]{AIR}}]
    There exists a bijection
    $$
    \begin{tikzcd}[row sep = 2mm]
        \mathrm{s}\tau\textrm{-}\mathrm{tilt}\,\Lambda \rar[rightarrow, "\sim"] \dar[phantom, "\rotatebox{90}{$\in$}"] & \mathrm{f\textrm{-}tors}\,\Lambda, \dar[phantom, "\rotatebox{90}{$\in$}"]\\
        M \rar[mapsto] & \mathrm{Fac}\,M,
    \end{tikzcd}
    $$
    where $\mathrm{Fac}\,M$ is a full subcategory of $\Lambda\mathrm{\textrm{-}mod}$ consisting of factor modules of a direct sum of copies of $M$.\\
\end{thm}

\subsection{$\tau$-Tilting finite algebras}

Demonet, Iyama and Jasso \cite{DIJ} introduced a new class of algebras, \textit{$\tau$-tilting finite algebras}, and found that $\tau$-tilting finiteness relates to brick finiteness and functorially finiteness of all the torsion classes.\\

\begin{dfn}
    An algebra $\Lambda$ is \textit{$\tau$-tilting finite} if there exist only finitely many basic support $\tau$-tilting modules up to isomorphism, that is, $\#\textrm{s$\tau$-tilt}\,\Lambda<\infty$.\\
\end{dfn}

\begin{dfn}
    A module $M\in\Lambda\textrm{-mod}$ is a \textit{brick} if the endomorphism algebra $\mathrm{End}_\Lambda(M)$ is isomorphic to $k$.\\
\end{dfn}

\begin{thm}[{\cite[Theorems 3.8 and 4.2]{DIJ}}]
    For an algebra $\Lambda$, the following are equivalent:
    \begin{enumerate}
        \setlength{\parskip}{0cm}
        \setlength{\itemsep}{0cm}
        \item $\Lambda$ is $\tau$-tilting finite.
        \item The number of isoclasses of bricks in $\Lambda\textrm{-}\mathrm{mod}$ is finite.
        \item Every torsion class in $\Lambda\textrm{-}{\mathrm{mod}}$ is functorially finite.\\
    \end{enumerate}
\end{thm}

Let $P_1,\ldots,P_t$ be all the nonisomorphic indecomposable projective $\Lambda$-modules and $K_0(\Lambda\textrm{-}\mathrm{proj})$ the Grothendieck group of $\Lambda\textrm{-}\mathrm{proj}$. Then the Grothendieck group $K_0(K^b(\Lambda\textrm{-}\mathrm{proj}))$ of $K^b(\Lambda\textrm{-}\mathrm{proj})$ is isomorphic to $K_0(\Lambda\textrm{-}\mathrm{proj})$ and both are free abelian groups of rank $t$ since they have a basis $\{[P_1],\ldots,[P_t]\}$ where each $[P_i]$ denotes the equivalence class of $P_i$ in the Grothendieck group. Henceforth, we identify these Grothendieck groups.\\

\begin{dfn}
    For a presilting complex $T=T_1\oplus\cdots\oplus T_r$ with each $T_i$ indecomposable, we define a \textit{$g$-cone} $C(T)\subset K_0(\Lambda\textrm{-}\mathrm{proj})\otimes_\mathbb{Z}\mathbb{R}$ of $T$ as the following:
    $$C(T):=\mathbb{R}_{\geq0}\cdot[T_1]+\cdots+\mathbb{R}_{\geq0}\cdot[T_r]\subset K_0(\Lambda\textrm{-}\mathrm{proj})\otimes_\mathbb{Z}\mathbb{R}.$$
\end{dfn}

\vspace{3mm}

\noindent The $\tau$-tilting finiteness of algebras is characterized in terms of $g$-cones as follows:\\

\begin{thm}[{\cite[Theorem 4.7]{A2}}]\label{asai}
    An algebra $\Lambda$ is $\tau$-tilting finite if and only if
    $$K_0(\Lambda\textrm{-}\mathrm{proj})\otimes_\mathbb{Z}\mathbb{R}=\bigcup_{T\in\mathrm{2\textrm{-}silt}\,\Lambda}C(T).$$
\end{thm}

\vspace{3mm}

\noindent In view of the structure of $g$-cones of two-term silting complexes, we consider the generalized class of algebras, \textit{$g$-tame algebras}, and compare $\tau$-tilting finite algebras and $g$-tame algebras with the classical representation types: representation finite algebras and tame algebras.\\

\begin{dfn}
    An algebra $\Lambda$ is said to be \textit{$g$-tame} if the union of $g$-cones of 2-term silting complexes is dense in $K_0(\Lambda\textrm{-}\mathrm{proj})\otimes_\mathbb{Z}\mathbb{R}$, that is,
    $$K_0(\Lambda\textrm{-}\mathrm{proj})\otimes_\mathbb{Z}\mathbb{R}=\overline{\bigcup_{T\in\textrm{2-silt}\,\Lambda}C(T)}.$$
\end{dfn}

\vspace{3mm}

\begin{thm}[{\cite[Theorem 4.1]{PYK}}]\label{pyk}
    A tame algebra is $g$-tame.\\
\end{thm}

\begin{thm}[See subsection 6.2 in \cite{EJR}]\label{tameblock}
    Tame blocks of group algebras of finite groups are $\tau$-tilting finite.\\
\end{thm}

\noindent By the above theorems, the relationship among representation types is as follows:
$$
\begin{tikzcd}[column sep = 15mm]
    \textrm{representation finite algebras} \rar[Rightarrow] \dar[Rightarrow] & \textrm{tame algebras}, \dar[Rightarrow, "\textrm{Theorem} \ref{pyk}"] \\
    \textrm{$\tau$-tilting finite algebras} \rar[Rightarrow, "\textrm{Theorem} \ref{asai}"] & \textrm{$g$-tame algebras}.
\end{tikzcd}
$$
In particular, we have the following hierarchy of representation types of blocks of group algebras of finite groups:
$$
\begin{tikzcd}
    \textrm{rep. fin. blocks} \rar[Rightarrow] & \textrm{tame blocks} \rar[Rightarrow, "\textrm{Theorem} \ref{tameblock}"] &[10mm] \textrm{$\tau$-tilt. fin. blocks} \rar[Rightarrow, "\textrm{Theorem} \ref{asai}"] &[10mm] \textrm{$g$-tame blocks}.
\end{tikzcd}
$$
Note that the converse of each implication above does not hold.\\
\indent We can find that blocks are representation finite or tame by looking at their defect groups as the following:\\

\begin{thm}[See THEOREM in Introduction in {\cite{E}}]\label{repblock}
    Let $B$ be a block of a group algebra $kG$ of a finite group $G$ and $D$ a defect group of $B$. Then the following hold:
    \begin{enumerate}
        \setlength{\parskip}{0cm}
        \setlength{\itemsep}{0cm}
        \item $B$ is representation finite if and only if $D$ is cyclic.
        \item $B$ is representation infinite and tame if and only if $p=2$ and $D$ is isomorphic to a dihedral, semidihedral or generalized quaternion group.\\
    \end{enumerate}
\end{thm}

\noindent It is natural to wonder what kinds of subgroups control $\tau$-tilting finiteness and $g$-tameness of group algebras or their blocks. We still do not know anything about what determines $g$-tameness of group algebras or their blocks, but we found that to treat $\tau$-tilting finiteness of group algebras, we should consider so-called \textit{$p$-hyperfocal subgroups}. Denote by $O^p(G)$ the smallest normal subgroup of $G$ such that its quotient is a $p$-group.\\

\begin{dfn}
    A \textit{$p$-hyperfocal subgroup} of a finite group $G$ is the intersection of a Sylow $p$-subgroup and $O^p(G)$.\\
\end{dfn}

\begin{prop}[{\cite[Proposition 2.15]{HK}}]\label{tfhyp}
    Let $R$ be a $p$-hyperfocal subgroup of a finite group $G$. Then a group algebra $kG$ is $\tau$-tilting finite if one of the following holds:
    \begin{enumerate}
        \setlength{\parskip}{0cm}
        \setlength{\itemsep}{0cm}
        \item $R$ is cyclic.
        \item $p=2$ and $R$ is isomorphic to a dihedral, semidihedral or generalized quaternion group.
    \end{enumerate}
\end{prop}

\vspace{3mm}

\noindent In \cite{HK}, we conjectured that the converse of Proposition \ref{tfhyp} also holds, and verified that it is true in the case $G=P\rtimes H$, where $P$ is an abelian $p$-group and $H$ is an abelian $p'$-group acting on $P$. In Section 4, we will see that the conjecture also holds for a group algebra $k[(\mathbb{Z}/m\mathbb{Z})^n\rtimes H]$ under the assumption that $p^l\geq n$, $H$ is a $p'$-subgroup of $\mathfrak{S}_n$, and the action of $H$ permutes the entries of $(\mathbb{Z}/m\mathbb{Z})^n$.\\

\section{Cartan matrices and $\tau$-tilting finiteness}

In this section, we discuss the relationship between Cartan matrices and $\tau$-tilting finiteness.\\
\indent Let $\Lambda$ be an algebra and $P_1,\ldots, P_t$ all the nonisomorphic indecomposable projective $\Lambda$-modules. When $\Lambda$ is selfinjective, we can define the Nakayama permutation $\nu\in\mathfrak{S}_t$ of $\Lambda$ such that $P_i\cong D\mathrm{Hom}_\Lambda(P_{\nu(i)},\Lambda)$ as a $\Lambda$-module for every $1\leq i\leq t$.\\

\begin{dfn}
    A \textit{Cartan matrix} $C_\Lambda$ of $\Lambda$ is the $t\times t$ matrix whose $(i,j)$-entry is $\mathrm{dim}_k\mathrm{Hom}_\Lambda(P_i,P_j)$.\\
\end{dfn}

\begin{prop}\label{tinf}
    Assume that $\Lambda$ is weakly symmetric.
    \begin{enumerate}
        \setlength{\parskip}{0cm}
        \setlength{\itemsep}{0cm}
        \item If $\Lambda$ is $\tau$-tilting finite, then the Cartan matrix $C_\Lambda$ of $\Lambda$ is positive definite.
        \item If $\Lambda$ is $g$-tame, then the Cartan matrix $C_\Lambda$ of $\Lambda$ is positive semidefinite.
    \end{enumerate}
\end{prop}
\begin{proof}
    (a) Since all entries of $C_\Lambda$ are integers, it suffices to show that $v^\top C_\Lambda v>0$ for any nonzero vector $v\in\mathbb{Z}^t$. The basis $[P_1],\ldots,[P_t]$ gives a natural identification of $K_0(\Lambda\textrm{-proj})$ with $\mathbb{Z}^t$. By $\tau$-tilting finiteness of $\Lambda$ and Theorem \ref{asai}, $v$ is in some $g$-cone of a basic two-term silting complex $T=T_1\oplus\cdots\oplus T_t$ with each $T_i$ indecomposable. By \cite[Theorem 2.27]{AI}, the vectors $[T_1],\ldots,[T_t]$ form a $\mathbb{Z}$-basis of $K_0(\Lambda\textrm{-proj})$. Hence there exist nonnegative integers $a_1,\ldots,a_t$ such that $v=\sum_{i=1}^t a_i[T_i]$. Let $T':=T_1^{\oplus a_1}\oplus\cdots\oplus T_t^{\oplus a_t}$. Since $\Lambda$ is weakly symmetric, a complex $T'$ is tilting by \cite[Theorem A.4]{Ai}. Therefore we have
    $$v^\top C_\Lambda v=\sum_{i\in\mathbb{Z}}(-1)^i\textrm{dim}_k\textrm{Hom}_{K^b(\Lambda\textrm{-proj})}(T',T'[i])=\textrm{dim}_k\textrm{End}_{K^b(\Lambda\textrm{-proj})}(T')>0.$$
    (b) In the proof of (a), we showed that $v^\top C_\Lambda v>0$ if $v\in\mathbb{Z}^t\setminus \{0\}$ is in some $g$-cone. It is also true that $v^\top C_\Lambda v\geq 0$ if $v\in\mathbb{R}^t$ is in some $g$-cone, because all $g$-cones are spanned by vectors with integer entries. Since $\Lambda$ is $g$-tame, $v^\top C_\Lambda v\geq 0$ holds for all vectors $v$ in the dense subset in $\mathbb{R}^t$. Therefore, $v^\top C_\Lambda v\geq 0$ holds for all vectors $v\in\mathbb{R}^t$.
\end{proof}

\vspace{3mm}

\begin{rem}
    The converse of Proposition \ref{tinf} does not hold. A counterexample to the converse of (a) is any $\tau$-tilting infinite group algebra because Cartan matrices of group algebras are positive definite. A counterexample to the converse of (b) is the following quiver algebra
    $$
    \begin{tikzcd}
        1 \ar[loop left, "a"] \rar[bend left=10] \rar[bend left=20] \rar[bend left=30, "b_i"] & 2 \ar[loop right, "c"] \lar[bend left=10] \lar[bend left=20] \lar[bend left=30, "d_i"]
    \end{tikzcd}
    $$
    with relations
    $$a^2=d_ib_i,\, c^2=b_id_i,\, ad_i=b_ia=b_id_j=cb_i=d_ib_j=d_ic=0\; (1\leq i\neq j \leq 3).$$
\end{rem}

\vspace{3mm}

\indent Thanks to Proposition \ref{tinf}, we can prove $\tau$-tilting infiniteness of weakly symmetric algebras by calculating Cartan matrices. In case $\Lambda$ is selfinjective, we have a similar assertion to Proposition \ref{tinf}.\\

\begin{prop}\label{tinfsi}
    Assume that $\Lambda$ is selfinjective. Let $\nu\in\mathfrak{S}_t$ be the Nakayama permutation of $\Lambda$. If there exists a nonzero vector $v\in\mathbb{Z}^t$ such that $v^\top C_\Lambda v\leq 0$ and $v$ is invariant under the action of $\nu$, the permutation of entries, then $\Lambda$ is $\tau$-tilting infinite.
\end{prop}
\begin{proof}
    Assume that $\Lambda$ is $\tau$-tilting finite and there exists a nonzero vector $v\in\mathbb{Z}^t$ such that $v^\top C_\Lambda v\leq 0$ and $v$ is invariant under the action of $\nu$. As in the proof of Proposition \ref{tinf} (a), we can take a nonzero two-term silting complex $T'$ such that $v=[T']$. Since $v$ is invariant under the action of $\nu$ and two-term silting complexes are uniquely determined by their $g$-vectors (\cite[Theorem 5.5]{AIR}), $T'$ is invariant under the Nakayama functor. Hence $T'$ is tilting by \cite[Theorem A.4]{Ai}. Therefore, we can show that $v^\top C_\Lambda v>0$ in the same way as Proposition \ref{tinf} (a), which leads to a contradiction.
\end{proof}

\vspace{2mm}

\section{When $k[(\mathbb{Z}/m\mathbb{Z})^n\rtimes\mathfrak{S}_n]$ is $\tau$-tilting finite}
In this section, we will determine when the group algebras over generalized symmetric groups are $\tau$-tilting finite.\\
\indent First, we consider when the group algebra $k[(\mathbb{Z}/p^\ell\mathbb{Z})^n\rtimes H]$ is $\tau$-tilting finite, where $\ell$ is a nonnegative integer such that $p^\ell\geq n$, $H$ is a subgroup of the symmetric group $\mathfrak{S}_n$, and the action of $H$ to $(\mathbb{Z}/p^\ell\mathbb{Z})^n$ is the permutation of entries. The group algebra $k[(\mathbb{Z}/p^\ell\mathbb{Z})^n\rtimes H]$ is isomorphic to the skew group algebra
$$\Lambda:= k[x_1,\ldots,x_n]/(x_1^{p^\ell},\ldots,x_n^{p^\ell})\rtimes H$$
with the $k$-basis $\{x_1^{i_1}\cdots x_n^{i_n}\sigma\mid 0\leq i_1,\ldots,i_n<p^\ell,\sigma\in H\}$ and the multiplication defined by    $$(f(x_1,\ldots,x_n)\sigma)\cdot(g(x_1,\ldots,x_n)\tau)=f(x_1,\ldots,x_n)g(x_{\sigma^{-1}(1)},\ldots,x_{\sigma^{-1}(n)})\sigma\tau$$
for any $f,g\in k[x_1,\ldots,x_n]/(x_1^{p^\ell},\ldots,x_n^{p^\ell})$ and $\sigma,\tau\in H$. For convenience, we write ${^\sigma f} (x_1,\ldots,x_n):=f(x_{\sigma^{-1}(1)},\ldots,x_{\sigma^{-1}(n)})$ for any polynomial $f$ and $\sigma\in H$. The skew group algebra $\Lambda$ can be given a grading by $\mathrm{deg}(x_i)=1$ for $1\leq i\leq n$ and $\mathrm{deg}(\sigma)=0$ for $\sigma\in H$.\\

\begin{lem}\label{Jac}
    The Jacobson radical of $\Lambda$ contains the positive degree part of $\Lambda$.
\end{lem}
\begin{proof}
    The assertion holds since the positive degree part of $\Lambda$ is a nilpotent ideal.
\end{proof}

\vspace{3mm}

\indent We denote by $E_i$ the $i$-th elementary symmetric polynomial, by $I$ the ideal of $k[x_1,\dots,x_n]$ generated by $E_1,\ldots,E_n$, and by $\mathcal{C}:=k[x_1,\ldots,x_n]/I$ the coinvariant algebra. The skew group algebra $\mathcal{C}\rtimes H$ is the key algebra in examining the support $\tau$-tilting modules over $k[(\mathbb{Z}/p^\ell\mathbb{Z})^n\rtimes H]$.\\

\begin{prop}\label{stauc}
    If $p^l\geq n$, then there exists an isomorphism as partially ordered sets
    $$\mathrm{s}\tau\textrm{-}\mathrm{tilt}\,k[(\mathbb{Z}/p^\ell\mathbb{Z})^n\rtimes H]\cong\mathrm{s}\tau\textrm{-}\mathrm{tilt}\,\mathcal{C}\rtimes H.$$
\end{prop}
\begin{proof}
    By Lemma \ref{Jac}, $E_1,\ldots,E_n$ are in both the center of $\Lambda$ and the Jacobson radical of $\Lambda$. Thus, we have an isomorphism as partially ordered sets
    $$\mathrm{s}\tau\textrm{-}\mathrm{tilt}\,k[(\mathbb{Z}/p^\ell\mathbb{Z})^n\rtimes H]\cong\mathrm{s}\tau\textrm{-}\mathrm{tilt}\,\Lambda\cong\mathrm{s}\tau\textrm{-}\mathrm{tilt}\,\Lambda/I\Lambda$$
    by \cite[Theorem 11]{EJR}. For any integer $1\leq j\leq n$, $x_j^n$ belongs to $I$ because
    \begin{equation}\label{eqxn}
        0=\prod_{i=1}^n(x_j-x_i)=x_j^n+\sum_{i=1}^n(-1)^iE_ix_j^{n-i}.
    \end{equation}
    By the assumption $p^\ell\geq n$, $x_j^{p^\ell}$ also belongs to $I$. Therefore,
    $$\Lambda/I\Lambda\cong k[x_1,\ldots,x_n]/(x_1^{p^\ell},\ldots,x_n^{p^\ell},E_1,\ldots,E_n)\rtimes H\cong \mathcal{C}\rtimes H.$$
\end{proof}

\vspace{3mm}

\indent The algebra $\Lambda/I\Lambda\cong\mathcal{C}\rtimes H$ is also graded because $I\Lambda$ is the ideal generated by homogeneous elements of $\Lambda$. As well as $\Lambda$, the Jacobson radical of $\mathcal{C}\rtimes H$ contains the positive degree part of $\mathcal{C}\rtimes H$ and the quotient algebra of $\mathcal{C}\rtimes H$ by its positive degree part is isomorphic to the group algebra $kH$. Hence, the primitive orthogonal idempotents of $\mathcal{C}\rtimes H$ are exactly those of the quotient algebra $kH$. We denote the complete set of simple modules over $kH$ by $\{S_\lambda\mid \lambda\in\mathrm{IBr}\,H\}$ and the projective cover of $S_\lambda$ by $P_\lambda$. Then the complete set of simple $\mathcal{C}\rtimes H$-modules is also the set $\{S_\lambda\mid\lambda\in\mathrm{IBr}\,H\}$, where we consider $S_\lambda$ as a $\mathcal{C}\rtimes H$-module under the natural embedding $kH\mathrm{\textrm{-}mod}\rightarrow\mathcal{C}\rtimes H\mathrm{\textrm{-}mod}$, and we denote the projective cover (resp. injective hull) of $S_\lambda$ over $\mathcal{C}\rtimes H$ by $\widetilde{P}_\lambda$ (resp. $\widetilde{I}_\lambda$). In particular, we denote the trivial (resp. sign) $kH$-module by $S_\mathrm{triv}$ (resp. $S_\mathrm{sgn}$). Note that tensoring simple $kH$-modules with $S_\mathrm{sgn}$ induces an involution on $\mathrm{IBr}\,H$, denoted by $(-)^*$, that is, it follows that $S_\lambda\otimes S_\mathrm{sgn}\cong S_{\lambda^*}$. It also follows that $P_\lambda\otimes S_\mathrm{sgn}\cong P_{\lambda^*}$.\\

\subsection{Selfinjectivity of $\mathcal{C}\rtimes H$}
In this subsection, we show that $\mathcal{C}\rtimes H$ is selfinjective. It is well-known (for example, see section I\hspace{-1.2pt}I-H in \cite{Ar}) that the coinvariant algebra $\mathcal{C}$ has a $k$-basis 
$$\{x_1^{i_1}\cdots x_n^{i_n}\in\mathcal{C}\mid 0\leq i_j\leq n-j\;(1\leq\forall j\leq n)\}.$$
Hence, we have a $k$-basis of $\mathcal{C}\rtimes H$
$$\mathcal{B}:=\{x_1^{i_1}\cdots x_n^{i_n}\sigma\in\mathcal{C}\rtimes H\mid 0\leq i_j\leq n-j\;(1\leq\forall j\leq n),\;\sigma\in H\}.$$
It is known that the maximum degree part of $\mathcal{C}$ is one dimensional and spanned by the monomial $\Delta:=x_1^{n-1}x_2^{n-2}\cdots x_{n-1}$, and hence isomorphic to $S_\mathrm{sgn}$ as a $kH$-module (for example, see (I.7) and (I.8) in \cite{GP}). Thus, it follows that $\sigma\Delta=\mathrm{sgn}(\sigma)\Delta\sigma$ in $\mathcal{C}\rtimes H$ for any $\sigma\in H$.\\

\begin{dfn}
    We define a $k$-linear map $\varphi:\mathcal{C}\rtimes H\rightarrow k$ by $\varphi(\alpha):=\delta_{\alpha,\Delta}\;(\alpha\in\mathcal{B})$ and a bilinear form $\langle-,-\rangle:\mathcal{C}\rtimes H\otimes\mathcal{C}\rtimes H\rightarrow k$ by $\langle\alpha,\beta\rangle:=\varphi(\alpha\beta)\;(\alpha,\beta\in\mathcal{C}\rtimes H)$.\\
\end{dfn}

\begin{prop}\label{selfinj}
    The bilinear form $\langle-,-\rangle:\mathcal{C}\rtimes H\otimes\mathcal{C}\rtimes H\rightarrow k$ is associative and nondegenerate. In particular, $\mathcal{C}\rtimes H$ is selfinjective.
\end{prop}
\begin{proof}
    It suffices to show that $\langle-,-\rangle$ is nondegenerate because associativity is obvious by the definition.\\
    \indent First, we show that $(x_1^{n-1}x_2^{n-2}\cdots x_i^{n-i})x_i\in I$ for all $1\leq i\leq n$ by induction on $i$. The assertion holds for $i=1$ by \eqref{eqxn}. Assume $i>1$ and let
    \begin{equation}\label{eqab1}
        \prod_{j=1}^{i-1}(t-x_j)=\sum_{j=0}^{i-1}a_jt^j,\;\prod_{j=i}^n(t-x_j)=\sum_{j=0}^{n-i+1}b_jt^j.
    \end{equation}
    Then,
    \begin{equation}\label{eqab2}
        \sum_{j=0}^{i-1}a_jt^j\sum_{j=0}^{n-i+1}b_jt^j=\prod_{j=1}^{n}(t-x_j)=\sum_{j=0}^{n}(-1)^{n-j}E_{n-j}t^j.
    \end{equation}
    By comparing the coefficients in the left and right hand side of \eqref{eqab2}, it follows inductively that for all $0\leq j\leq n-i$, $b_j$ is generated by $a_0,\ldots,a_{i-2}\in k[x_1,\ldots,x_{i-1}]$ in $\mathcal{C}$ since $a_{i-1}=b_{n-i+1}=1$. In particular, for $1\leq j\leq n-i$, since $\mathrm{deg}(b_j)>0$, $b_j$ can be written as a sum of monomials of positive degree with variables $x_1,\ldots,x_{i-1}$ in $\mathcal{C}$. Substituting $t=x_i$ in \eqref{eqab1}, we get
    $$0=x_i^{n-i+1}+\sum_{j=0}^{n-i}b_jx_i^j,$$
    and hence we have
    \begin{equation}\label{eqlem}
        (x_1^{n-1}x_2^{n-2}\cdots x_i^{n-i})x_i=-(x_1^{n-1}x_2^{n-2}\cdots x_{i-1}^{n-i+1})\sum_{j=0}^{n-i}b_jx_i^j.
    \end{equation}
    By induction hypothesis, the right hand side of \eqref{eqlem} vanishes in $\mathcal{C}$. Therefore, it follows that $(x_1^{n-1}x_2^{n-2}\cdots x_i^{n-i})x_i\in I$, establishing the induction step.\\
    \indent We define the lexicographic order on monomials as follows:
    $$x_1^{i_1}\cdots x_n^{i_n}< x_1^{i'_1}\cdots x_n^{i'_n}\Leftrightarrow i_s< i'_s\;(s:=\min\{1\leq j\leq n\mid i_j\neq i'_j\}).$$
    It follows that any monomial greater than $\Delta$ with respect to the lexicographic order is equal to zero in $\mathcal{C}$ by the above discussion. Take arbitrary $0\neq\alpha\in\mathcal{C}\rtimes H$. Among monomials appearing when we write $\alpha$ as a linear combination of elements in $\mathcal{B}$, we denote by $f$ the smallest monomial with respect to the lexicographic order. Then there exists a monomial $g$ such that $gf=\Delta$. For any monomial $h>f$, we have $gh>gf=\Delta$ and so $gh$ vanishes in $\mathcal{C}$. Thus, we can write
    $$g\alpha=\sum_{\sigma\in H}c_\sigma \Delta\sigma\;(c_\sigma\in k)$$
    and take $\sigma_0\in H$ such that $c_{\sigma_0}\neq0$. Therefore,
    $$\langle\sigma_0^{-1}g,\alpha\rangle=\varphi\left(\sigma_0^{-1}\sum_{\sigma\in H}c_\sigma\Delta\sigma\right)=\varphi\left(\sum_{\sigma\in H}c_{\sigma}\mathrm{sgn}(\sigma_0^{-1})\Delta\sigma_0^{-1}\sigma\right)=c_{\sigma_0}\mathrm{sgn}(\sigma_0^{-1})\neq0,$$
    which implies that $\langle-,-\rangle$ is nondegenerate.
\end{proof}

\vspace{3mm}

\begin{prop}\label{Nakayama}
    For $\lambda\in\mathrm{IBr}\,H$, we have $\widetilde{P}_\lambda\cong\widetilde{I}_{\lambda^*}$.
\end{prop}
\begin{proof}
    We can take a primitive idempotent $e_\lambda$ of $kH$ such that $(\mathcal{C}\rtimes H)e_\lambda=\widetilde{P}_\lambda$. Since $S_\lambda$ is the socle of $P_\lambda\subset kH$, we can regard $S_\lambda$ as a subspace of $kH$. The subspace $\{\Delta x\in\mathcal{C}\rtimes H\mid x\in S_\lambda\}$ of $\widetilde{P}_\lambda=(\mathcal{C}\rtimes H)e_\lambda=\{fx\in\mathcal{C}\rtimes H\mid f\in\mathcal{C},x\in P_\lambda\}$ is a $\mathcal{C}\rtimes H$-module and isomorphic to a simple $\mathcal{C}\rtimes H$-module $S_\mathrm{sgn}\otimes S_\lambda$ since $\alpha\Delta=0$ if $\mathrm{deg}(\alpha)>0$ and $\sigma\Delta=\mathrm{sgn}(\sigma)\Delta\sigma$ for $\sigma\in H$. Therefore, we have $\mathrm{soc}(\widetilde{P}_\lambda)\cong S_{\lambda^*}$, which implies $\widetilde{P}_\lambda\cong\widetilde{I}_{\lambda^*}$.
\end{proof}

\vspace{3mm}

\begin{cor}
    If $p=2$, then $\mathcal{C}\rtimes H$ is weakly symmetric.
\end{cor}
\begin{proof}
    The assertion follows since $S_\mathrm{sgn}\cong S_\mathrm{triv}$ when $p=2$.
\end{proof}

\vspace{3mm}

\subsection{When $\mathcal{C}\rtimes H$ is $\tau$-tilting infinite}
In this subsection, we first compute the Cartan matrix of $\mathcal{C}\rtimes H$ and then establish the sufficient condition for $\tau$-tilting infiniteness of $\mathcal{C}\rtimes H$.\\

\begin{prop}\label{Cartan}
    For any $\lambda,\mu\in\mathrm{IBr}\,H$, the $(\lambda,\mu)$-entry of the Cartan matrix $C_{\mathcal{C}\rtimes H}$ of $\mathcal{C}\rtimes H$ is equal to $(\mathfrak{S}_n:H)\cdot\mathrm{dim}_k\widetilde{P}_\lambda\cdot\mathrm{dim}_k\widetilde{P}_\mu$. In particular, the Cartan matrix $C_{\mathcal{C}\rtimes H}$ has rank one.
\end{prop}
\begin{proof}
    By \cite[Theorem 1.4]{M}, it follows that $[\mathcal{C}]=[k\mathfrak{S}_n]$ in the Grothendieck group $K_0(k\mathfrak{S}_n\mathrm{\textrm{-}mod})$ of $k\mathfrak{S}_n\mathrm{\textrm{-}mod}$, and hence we have $[\mathcal{C}]=(\mathfrak{S}_n:H)[kH]$ in the Grothendieck group $K_0(kH\mathrm{\textrm{-}mod})$ of $kH\mathrm{\textrm{-}mod}$. Thus, we have the following isomorphisms as $kH$-modules:
    $$\widetilde{P}_\mu\cong\mathcal{C}\otimes P_\mu\cong kH^{\oplus (\mathfrak{S}_n:H)}\otimes P_\mu\cong kH^{\oplus (\mathfrak{S}_n:H)\cdot\mathrm{dim}_kP_\mu}.$$
    Therefore, it follows that
    $$
    \begin{array}{ccl}
        \mathrm{dim}_k\mathrm{Hom}_{\mathcal{C}\rtimes H}(\widetilde{P}_\lambda,\widetilde{P}_\mu) & = & \mathrm{dim}_k\mathrm{Hom}_{\mathcal{C}\rtimes H}(\mathcal{C}\rtimes H\otimes_{kH}P_\lambda,\widetilde{P}_\mu), \\
         & = & \mathrm{dim}_k\mathrm{Hom}_{kH}(P_\lambda,\widetilde{P}_\mu), \\
         & = & \mathrm{dim}_k\mathrm{Hom}_{kH}(P_\lambda,kH)\cdot (\mathfrak{S}_n:H)\cdot\mathrm{dim}_kP_\mu, \\
         & = & (\mathfrak{S}_n:H)\cdot\mathrm{dim}_kP_\lambda\cdot \mathrm{dim}_kP_\mu.
    \end{array}
    $$
\end{proof}

\vspace{3mm}

\begin{prop}\label{main0}
    If $\#\mathrm{IBr}\,H\geq\min\{p,3\}$, then $\mathcal{C}\rtimes H$ is $\tau$-tilting infinite.
\end{prop}
\begin{proof}
    By assumption, there exist two simple $kH$-modules $S_\lambda\ncong S_\mu$ such that $S_\mathrm{sgn}\otimes S_\lambda\ncong S_\mu$. For $\chi\in\mathrm{IBr}\,H$, we denote by $v_\chi\in\mathbb{Z}^{\#\mathrm{IBr}\,H}$ the vector with entry 1 at $\chi$-coordinate and entry zero elsewhere. Let
    $$v:=\mathrm{dim}_kP_\mu\cdot(v_\lambda+v_{\lambda^*})-\mathrm{dim}_kP_\lambda\cdot(v_\mu+v_{\mu^*})\in\mathbb{Z}^{\#\mathrm{IBr}\,H}.$$
    The vector $v$ is nonzero and invariant under the action of Nakayama permutation $(-)^*$ (see Proposition \ref{Nakayama}). Moreover, we have $C_{\mathcal{C}\rtimes H}\cdot v=0$ because each row of $C_{\mathcal{C}\rtimes H}$ is a multiple of a vector $(\mathrm{dim}_kP_\chi)_{\chi\in\mathrm{IBr}\,H}$ by Proposition \ref{Cartan} and
    $$\mathrm{dim}_kP_\mu\cdot(\mathrm{dim}_kP_\lambda+\mathrm{dim}_kP_{\lambda*})-\mathrm{dim}_kP_\lambda\cdot(\mathrm{dim}_kP_\mu+\mathrm{dim}_kP_{\mu*})=0.$$
    Therefore, $\mathcal{C}\rtimes H$ is $\tau$-tilting infinite by Propositions \ref{tinfsi} and \ref{selfinj}.
\end{proof}

\vspace{3mm}

\subsection{Main results}
We are now ready to prove our main results. In this subsection, we give a criterion for $\tau$-finiteness of the group algebra of the group $(\mathbb{Z}/m\mathbb{Z})^n\rtimes H$ in terms of its $p$-hyperfocal subgroup. We write $m=p^lm'$ where $l,m'$ are integers and $p$ does not divide $m'$. Then we have the algebra surjection $k[(\mathbb{Z}/m\mathbb{Z})^n\rtimes H]\twoheadrightarrow k[(\mathbb{Z}/p^l\mathbb{Z})^n\rtimes H]$ induced by the natural group surjection $(\mathbb{Z}/m\mathbb{Z})^n\rtimes H\twoheadrightarrow(\mathbb{Z}/p^l\mathbb{Z})^n\rtimes H$. Thus, by Propositions \ref{stauc} and \ref{main0}, we obtain the following corollary:\\

\begin{cor}\label{main1}
    If $p^l\geq n$ and $\#\mathrm{IBr}\,H\geq\min\{p,3\}$, then the group algebra $k[(\mathbb{Z}/m\mathbb{Z})^n\rtimes H]$ is $\tau$-tilting infinite.\\
\end{cor}

\indent To state the main result, we compute the $p$-hyperfocal subgroup of $(\mathbb{Z}/m\mathbb{Z})^n\rtimes H$ when $H$ is a $p'$-group.\\

\begin{prop}\label{rk}
    Assume that $H$ is a $p'$-subgroup of $\mathfrak{S}_n$. Denote by $R$ the $p$-hyperfocal subgroup of $(\mathbb{Z}/m\mathbb{Z})^n\rtimes H$. Then the sum of ranks of two groups $R$ and $C_{(\mathbb{Z}/m\mathbb{Z})^n}(H)$ is equal to $n$.
\end{prop}
\begin{proof}
    By applying \cite[Proposition 3.5 and Lemma 3.9]{HK} to $$(\mathbb{Z}/m\mathbb{Z})^n\rtimes H=(\mathbb{Z}/p^l\mathbb{Z})^n\rtimes ((\mathbb{Z}/m'\mathbb{Z})^n\rtimes H),$$
    we have
    $$(\mathbb{Z}/p^l\mathbb{Z})^n=R\times C_{(\mathbb{Z}/p^l\mathbb{Z})^n}((\mathbb{Z}/m'\mathbb{Z})^n\rtimes H)=R\times C_{(\mathbb{Z}/p^l\mathbb{Z})^n}(H).$$
    Since $C_{(\mathbb{Z}/p^l\mathbb{Z})^n}(H)\cong(\mathbb{Z}/p^l\mathbb{Z})^{n'}$ for some $n'$ and two groups $C_{(\mathbb{Z}/p^l\mathbb{Z})^n}(H),\,C_{(\mathbb{Z}/m\mathbb{Z})^n}(H)$ has the same rank, the assertion follows. 
\end{proof}

\vspace{3mm}

\indent By Corollary \ref{main1} and Proposition \ref{rk}, we can obtain
another example satisfying the conjecture in \cite{HK}, which says that $\tau$-tilting finiteness of group algebras is controlled by their $p$-hyperfocal subgroups.\\

\begin{thm}\label{main2}
    Assume that $p^l\geq n$ and $H$ is a $p'$-subgroup of $\mathfrak{S}_n$. Denote by $R$ the $p$-hyperfocal subgroup of $(\mathbb{Z}/m\mathbb{Z})^n\rtimes H$. Then $k[(\mathbb{Z}/m\mathbb{Z})^n\rtimes H]$ is $\tau$-tilting finite if and only if $R$ has rank $\leq1$.
\end{thm}
\begin{proof}
    The ``if'' part follows from Proposition \ref{tfhyp}. To prove the ``only if'' part, we assume that $k[(\mathbb{Z}/m\mathbb{Z})^n\rtimes H]$ is $\tau$-tilting finite. By Corollary \ref{main1}, we have $\#\mathrm{IBr}\,H\leq2$, which implies that $H$ has at most two conjugacy classes since $H$ is a $p'$-group. Hence, $H$ is trivial or isomorphic to $\mathbb{Z}/2\mathbb{Z}$. Thus, $C_{(\mathbb{Z}/m\mathbb{Z})^n}(H)$ has rank $\geq n-1$. Therefore, by Proposition \ref{rk}, $R$ has rank $\leq 1$.
\end{proof}

\end{document}